\documentclass[11pt,oneside,pdftex]{amsart}
\usepackage{geometry}   
\usepackage{wrapfig}
\usepackage{color,amsmath,amsfonts,amssymb,amsthm,color,mathtools}
\usepackage{cancel}
\usepackage{tensor}
\usepackage{tikz-cd}
\usepackage{bbm}
\usepackage{mathrsfs}
\usepackage{subfiles}
\usepackage{fancyhdr}
\usepackage{graphicx}
\usepackage{enumitem}
\usepackage{hyperref}

\usepackage{ifthen}

\newtheorem{thm}{Theorem}[section]
\newtheorem{lemma}[thm]{Lemma}
\newtheorem{lem}[thm]{Lemma}
\newtheorem*{lemma*}{Lemma}

\newtheorem{cor}[thm]{Corollary}

\newtheorem{conj}[thm]{Conjecture}

\theoremstyle{definition}
\newtheorem{defn}[thm]{Definition}

\newtheorem{rem}[thm]{Remark}

\newtheorem{prob}[thm]{Problem}

\newcommand{\integers}{\mathbb{Z}}

\newcommand{\field}{\mathbb{F}}

\DeclareMathOperator{\stab}{Stab}

\DeclarePairedDelimiterX{\inner}[2]{\langle}{\rangle}{#1 , #2}
\DeclarePairedDelimiterX{\pres}[2]{\langle}{\rangle}{#1 \delimsize\vert #2}

\begin{document}

\title
[Residually finite groups that do not virtually have UPP]
{Residually finite groups that do not virtually have the unique product property}

\author{Naomi Bengi}
           \address{Dept. of Mathematics\\
                    Weizmann Institute of Science}
           \email{naomi.bengi@weizmann.ac.il}

\author[D.~T.~Wise]{Daniel T. Wise}
\email{daniel.wise@weizmann.ac.il}
\date{\today}
\thanks{}

\begin{abstract}
We construct  a finitely generated residually finite group $G$ with the property  that every finite index subgroup of $G$ contains a subgroup isomorphic to Promislow's group. Hence $G$ does not have a finite index subgroup with the unique product property.
\end{abstract}

\maketitle

\section{Introduction}
\label{sec:intro}

\begin{defn}[Persistently contains]
Let $P$ be a group.
A group $G$ \emph{persistently contains $P$}
if for every finite index subgroup $G'\subset G$,
there is a subgroup $H\subset G'$ with $H\cong P$.
\end{defn}

For instance, if $G$ has a finite index subgroup isomorphic to $\integers^r$ then $G$ persistently contains $\integers^r$.
Likewise for persistently containing a rank~$r$ free group.
Of course, if $G$ has no proper finite index subgroups, then $G$ persistently contains every subgroup of $G$,
so we will focus on the case where $G$ is residually finite.

An easy group with notable properties is \emph{Promislow's group}
which is presented by:
$$P = \langle x,y \ \mid \ x^{-1}y^2xy^2 \ , \ y^{-1}x^2yx^2 \rangle$$
The above presentation hides the simple structure of $P$. It has an index~4 subgroup isomorphic to $\integers^3$. In fact, $P$ arises as an amalgamated product, actually a `double': 
$$P\ \cong \ K\underset{T}{*} K$$
where $K$ is $\pi_1$ of the Klein-bottle, and $T\subset K$ is the index~2 subgroup isomorphic to  $\integers^2$. One quickly sees that $P$ is torsion-free and virtually $\integers^3$ from this viewpoint.

A group $G$ has  the \emph{unique product property} (UPP)
if for any nonempty finite subsets $A,B\subset G$,
there is a pair $(a,b)\in A\times B$ such that $ab\neq a'b'$ whenever $(a',b')\in A\times B -\{(a,b)\}$.
If $G$ satisfies the UPP then $G$ satisfies Kaplansky's ``zero-divisor conjecture'', which asserts that $\field G$ has no zero-divisors when $\field$ is a field and $G$ is torsion-free.  
Promislow's group was the second type of torsion-free group that failed to satisfy UPP, the first examples were discovered shortly before in \cite{RipsSegev87}.

Recently, Gardam showed that $P$ also fails to satisfy Kaplansky's ``unit conjecture'' \cite{Gardam2021} which asserts that the only units in $\field G$ are of the form $kg$ with $k\in \field-0$ and $g\in G$.

The  goal of this note is to show: 
\begin{thm}
    There is a finitely generated residually finite torsion-free group that persistently contains Promislow's group.
\end{thm}

This answers a question of \cite{KonkeRaimbaultDunfield_2016} who asked if every finitely generated residually finite torsion-free group is virtually diffuse.
\emph{Diffuse} is a property implying UPP that was studied by Bowditch \cite{Bowditch2000}.
Recently Ng also produced residually finite groups that are not virtually diffuse \cite{Ng2025}.

We are curious about what other groups are persistently contained in residually finite groups. Perhaps the following holds: 

\begin{conj}
    Let $H$ be a finitely generated residually finite group. Then there exists a finitely generated residually finite group $G$ that persistently contains $H$.
    \end{conj}

In a positive direction, perhaps:
\begin{prob}\label{prob:linear_virtually_UPP}
Let $G$ be a f.g.\ linear group.
Does $G$ have a finite index subgroup with UPP ?
\end{prob}

\section{Doubles}
\label{sec:doubles}

This text repeatedly uses the notion of a ``double of a group along a subgroup''.
For a subgroup $H\subset G$, 
we let $\underline G$ be a copy of $G$, so there is an isomorphism $G\rightarrow \underline G$ denoted by $g\mapsto \underline g$,
and let $\underline H$ be the image of $H$ in $\underline G$.
The \emph{double} of $G$ along $H$ is the amalgamated free product $G\underset{H=\underline H }{*} \underline G$, which amalgamates $G$ and $\underline G$ by identifying $H$ with $\underline H$ using this isomorphism.
We will simply use the notation
$G\underset{H}{*} \underline G$.

The most famous example of a double is the Klein bottle group $K= \integers\underset{2\integers}{*} \underline \integers$.
And as mentioned in the introduction, Promislow's group $P$ arises as the double $P = K\underset{T}{*} \underline K$.

We will frequently use the following:
\begin{lem}\label{lem:NFT_and_doubles}
If $A\subset B$ and $C=A\cap D$
then $A\underset{C}{*} \underline A$ injects in $B\underset{D}{*} \underline B$.
\end{lem}
\begin{proof}
This holds by the normal form theorem for amalgamated products \cite{LS77}.
Indeed, the hypothesis ensures that normal forms map to normal forms.
\end{proof}

For $z\in G$, \emph{centralizer} of $z$ is  $C_G(z) = \{ g \in G \mid zgz^{-1} = g \}$.

\begin{lemma}\label{lem:the_centralizer}
Let $N \subset H \subset F$ be groups with $[H:N] = 2$. 
Consider the doubles $L = H \underset{N}{*} \underline{H}$ and $G = F \underset{N}{*} \underline{F}$. By Lemma~\ref{lem:NFT_and_doubles} we can view $L$ as a subgroup of $G$.

Choose $h \in H - N$, and let  
$\ell = h \underline{h}^{-1} \in L$.
Then  $C_G(\ell )$ is an index-2 normal subgroup of $L$, independent of the choice of $h$.
Furthermore, $C_G(\ell )$ splits internally as a direct product $N \times \langle \ell \rangle$.
\end{lemma}

\begin{proof}
The Bass-Serre tree $A$ of $L=H\underset{N}{*} \underline H$ is a line.
Elements of $H-N$ and $\underline H - \underline {N}$ act by reflections, but $N$ fixes $A$.
Let $M$ be the index~2 subgroup of $L$
consisting of all translations.
Then $N\subset M$,
and $\ell \in M$ as $\ell =h\underline h^{-1}$ is the product of reflections.

First we show that $M = N \times \langle \ell  \rangle$.
Both $N$ and $\ell $ lie in $M$. That $\ell  = h \underline{h}^{-1}$ commutes with each  $n \in N$ holds since:
$$\underline{h}^{-1} n\underline{h} 
=
\underline{h}^{-1}\underline n\,\underline{h}
=
h^{-1} n h \ \ \ \implies \ \ \
(h\underline{h}^{-1})n = n(h \underline{h}^{-1}) $$


Since $\ell ^q$ is a nontrivial translation for $q\neq 0$,
we see that  $N \cap \langle \ell  \rangle$ is trivial. Thus,  $\langle N, \ell  \rangle \ \cong N \times \langle \ell  \rangle  \ \subseteq \ M$. Finally, as there are two orbits of vertices, the length~2 translation $\ell $ is a minimal length translation. Hence  any  $m\in M$ has the property that $m\ell ^q$ acts trivially for some $q\in \integers$,
and so $m\ell ^q\in N$.  Thus $M=N\times \langle \ell \rangle$.

We now show $M = C_G(\ell )$.
As $M \subseteq C_G(\ell )$ holds from  the above,  we show $C_G(\ell ) \subseteq M$.
 
Let $T$ be the Bass-Serre tree of 
$G=F\underset{N}{*} \underline{F}$.
Observe that $A$ embeds as a subtree of $T$.
Furthermore, $A$ is the axis of $s\underline s^{-1}$ for any $s \in H-N$,
and in particular $A$ is the axis of $\ell $.


We now show $C_G(\ell ) \subseteq \stab_G(A)$. Let $g \in C_G(\ell )$. The axis of
$g\ell g^{-1}$ is  $gA$. Since $g\ell g^{-1}=\ell $, their axes are identical, so $gA=A$. 

We now show $\stab_G(A) = L$.
We already know $L \subseteq \stab_G(A)$,
so it suffices to show
$\stab_G(A) \subseteq L $.
Suppose $gA=A$.
Let $v,\underline v$ be the vertices of $A$ stabilized by $F$ and $\underline F$.
Then for some $q\in \integers$, we have $g\ell ^qv=v$
or 
$g\ell ^q\underline v=\underline v$,
or both.
However, $\stab_G(v)\cap \stab_G(A)=H$,
as can be from the correspondence of edges at $v$ with cosets of $N$ in $F$.
Likewise $\stab_G(\underline v)\cap \stab_G(A)=\underline H$.
Thus $g\ell ^q\in L$ so $g\in L$.

In conclusion
$C_G(\ell ) \subseteq \stab_G(A) = L$
and hence $C_G(\ell )=C_L(\ell )=M$.
\end{proof}

\section{Preliminaries on residual finiteness}
\label{sec:profinit}

A subgroup $H\subset G$ is \emph{separable} if $H$ is the intersection of finite index subgroups of $G$.
This is equivalent to $H$ being a closed subset in the \emph{profinite topology} on $G$ which is the topology having a basis consisting of cosets of finite index subgroups of $G$.
The   standard observations follows from the definitions:
\begin{lem}\label{lem:basic_properties}
The profinite topology on $G$ is Hausdorff if $G$ is residually finite.

A homomorphism between groups is continuous in the profinite topology.
\end{lem}

The following was explicitly proven in \cite{BolerEvans1973}, but can be deduced from \cite{Baumslag63b}.

\begin{lem}\label{lem:rf_double}
Let $G$ be  residually finite group.
Let $H\subset G$ be separable.
Then the amalgamated product
$G\underset{H}{*} \underline G$ is r.f.
\end{lem}
\begin{proof}[Sketch]
Nontrivial elements of $G$ or $\underline G$, 
are handled via the homomorphism $G\underset{H}{*} \underline G \rightarrow G$, followed by the fact  that $G$ is residually finite. 
For $x\notin  G\cup \underline G$, 
consider the normal form $x_1x_2\cdots x_p$ of $x$ \cite{LS77}.
Separability of $H\subset G$
 provides a finite quotient $G\rightarrow \bar G$ separating each $x_iH$ from $H$.
The image of $x$ in  $\bar G \underset{{\bar H}}{*} \underline{\bar G}$ is nontrivial since its normal form maps to a nontrivial normal form. And $\bar G\underset{\bar H}{*} \underline{\bar G}$ is residually finite as it is virtually free.
\end{proof}

\begin{lemma}[Equalizer is Closed]\label{lem:equalizer_closed}
Let $f, g: X \to Y$ be continuous maps of topological spaces. If $Y$ is Hausdorff, then the equalizer $E = \{x \in X \mid f(x) = g(x)\}$ is a closed subset.
\end{lemma}
\begin{proof}
We show  $X - E$, is open. Let $x \in X -E$,
so $f(x) \neq g(x)$. As $Y$ is Hausdorff, there are disjoint  neighborhoods $U$ of $f(x)$ and $V$ of $g(x)$ in $Y$.

Let $W = f^{-1}(U) \cap g^{-1}(V)$. Then $W$ is an open neighborhood of $x$.

For any  $w \in W$, we have $f(w) \in U$ and $g(w) \in V$.  
Thus $f(w) \neq g(w)$ since $U\cap V=\emptyset$, hence $w \notin E$. Therefore, $W \subseteq X - E$. Thus $X - E$ is open.
\end{proof}

\begin{cor}\label{cor:retract_closed}
Let $G$ be a residually finite group.
Let $H\subset G$ be a retract of $G$,
in the sense that there is a homomorphism $r:G\rightarrow G$ with $H=r(G)$ and $r(h)=h$ for all $h\in H$.
Then $H\subset G$ is separable.
\end{cor}
\begin{proof}
This holds by Lemma~\ref{lem:equalizer_closed}
applied to the equalizer of 
$r:G\rightarrow G$ and  $id:G\rightarrow G$.
\end{proof}

\begin{cor}[Separability of Centralizers]\label{cor:centralizer_separable}
Let $G$ be a residually finite group and let $z \in G$. Then the centralizer $C_G(z)$ is a separable subgroup of $G$.
\end{cor}
\begin{proof}
The profinite topology on $G$ is Hausdorff since $G$ is residually finite. 

Observe that $C_G(z)$ is the equalizer of two continuous maps from $G$ to itself:
\begin{enumerate}
    \item The inner automorphism $\psi_z: G \to G$ given by $g \mapsto zgz^{-1}$. 
    \item The identity map $\mathrm{id}: G \to G$ given by $g \mapsto g$.
\end{enumerate}
By Lemma~\ref{lem:equalizer_closed}, $C_G(z)$ is a closed subset of $G$, and hence separable.
\end{proof}

\section{A family of doubles}
A  sequence  $\vec n$ is \emph{multiplicative} if $n_{i+1}$ is a multiple of $n_i$ for each $i$.

Let $F$ be the free group on $a,b$. 
Let $\vec n =(n_i)$ be a multiplicative sequence of natural numbers.
Let $H=H(\vec n)$ be the following subgroup:
\begin{equation}\label{eq:subgroup_H}
H = \langle a^{-i} b^{n_i} a^i  \ : \ 0 \leq i \rangle.
\end{equation}

We  also use the following finitely generated subgroups $H_k=H_k(\vec n)$, 
which are clarified geometrically below:
$$H_k=
\langle a^k, \ \ a^{-i} b^{n_i} a^i   \ : \  0\leq i < k \rangle. $$

\begin{figure}
    \centering
    \includegraphics[width=0.5\linewidth]{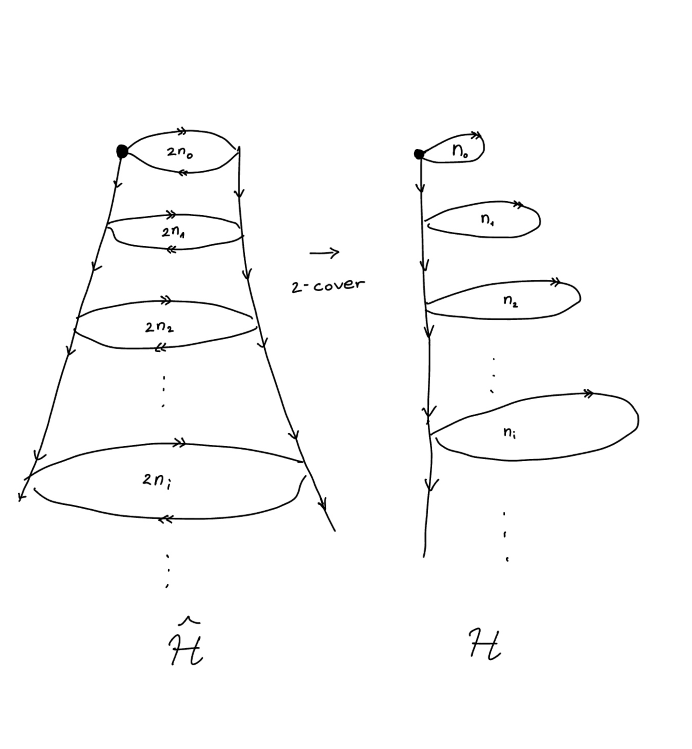}
    \caption{$\widehat{\mathcal{H}}$ is a double cover of $\mathcal{H}$}
    \label{fig:Hhat}
\end{figure}

\begin{figure}
    \centering
    \includegraphics[width=0.7\linewidth]{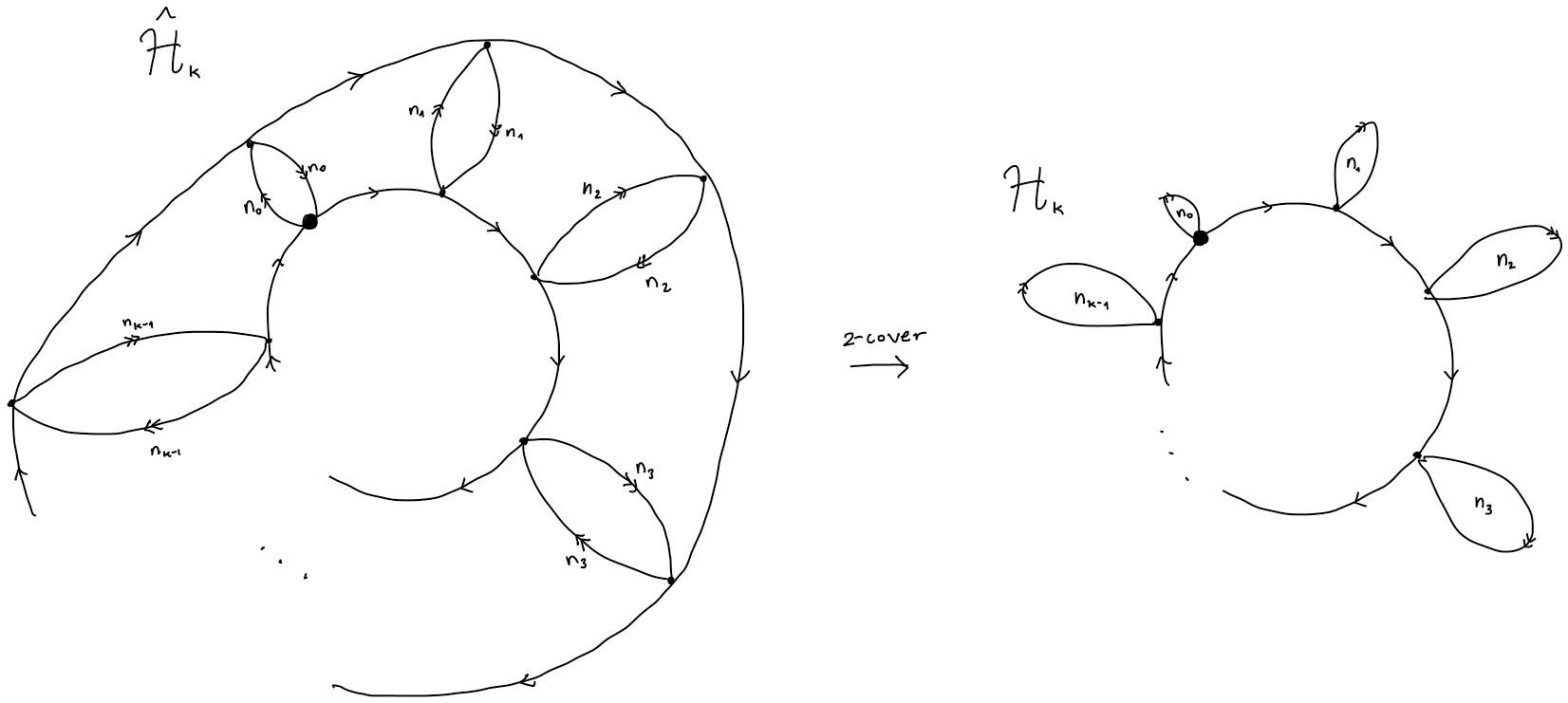}
    \caption{$\widehat{\mathcal{H}}_k$ is a double cover of $\mathcal{H}_k$}
    \label{fig:Hihat}
\end{figure}

We follow the viewpoint popularized by Stallings relating immersed graphs and subgroups of free groups \cite{Stallings83}. We do not mention the base-vertex below but it can be identified with a copy of $\{0\}$ when there is more than one vertex.

Let $\mathcal{F}$ be a bouquet of two circles directed and labeled by $a$ and $b$.
So $F = \pi_1\mathcal{F}$.

For $k \geq 0$, let $\mathcal{C}_k$ be the connected degree~$n_k$ cover of the $b$-circle in $\mathcal{F}$. So $\pi_1 \mathcal{C}_k = \langle b^{n_k}\rangle$.

Let $\mathcal{H}$ be the graph obtained from the ray $R = [0, \infty)$
by attaching  circles $\mathcal{C}_k$ along $k$, as in the right of Figure~\ref{fig:Hhat}.
So $H = \pi_1\mathcal{H}$.

Let $\mathcal{H}_k$ be the \emph{$k$-th summary} of $\mathcal{H}$ as illustrated at the right of Figure~\ref{fig:Hihat}. It is formed from $[0,k] / \{0\sim k\}$ by attaching a copy of $\mathcal{C}_j$ at $\{j\}$ for $0\leq j < k$.
The crucial consequence of multiplicativity of $\vec n$ is that 
$\mathcal H\rightarrow \mathcal F$ factors as 
$\mathcal H\rightarrow \mathcal H_k \rightarrow \mathcal F$ for each $k$.
Observe that
$H_k =\pi_1\mathcal H_k$ where the basepoint is bold.
 
These subgroups were introduced in
\cite{WiseDescending99} and further studied in \cite{ChongWise2022}
 which showed:
\begin{thm}\label{thm:H_separable}
Let $H$ and $H_k$ be associated to the multiplicative sequence $\vec n$.\\
Then $H=\cap_k H_k$.
Consequently,  $H$ is separable.
\end{thm}
\begin{proof}[Sketch of proof]
Let 
$\widehat {\mathcal F}$ be the based covering space of $\mathcal F$ with $\pi_1\widehat {\mathcal F} = H$,
and note that there is an embedding $\mathcal H \hookrightarrow \widehat {\mathcal F}$.
Define $\widehat {\mathcal F}_k$ similarly,
and note  $\mathcal H_k \hookrightarrow \widehat {\mathcal F}_k$ for each $k$.

Consider a path $\sigma$  representing an element of $F-H$, so $\sigma$ lifts to a non-closed path $\hat \sigma$ in $\widehat{\mathcal F}$.
It is here we use that $\mathcal H_k$ is a good ``summary''  of $\mathcal H$ in the sense that their large balls look the same for sufficiently large~$k$.
Because of this, 
the based lift of $\hat \sigma_k$ to $\widehat {\mathcal F}_k$ is also not closed for sufficiently large $k$. Hence $\sigma \notin H_k$.

 Each $H_k = \pi_1 \mathcal{H}_k$ is separable since it is a finitely generated subgroup of a free group. Hence  $H = \cap_k H_k$ is separable.
\end{proof}

We now describe index~2 subgroups $\widehat H=\widehat H(\vec n)$ of $H$, by building double covers of $\widehat {\mathcal H}$. 
Let $\widehat{\mathcal{H}}$ be the double cover of $\mathcal{H}$ obtained by gluing the connected double covers $\widehat{C_k}$ of $C_k$, to two rays along  copies of $k$, as in Figure~\ref{fig:Hhat}.
Let $\widehat{H} = \pi_1\widehat{\mathcal{H}}$. So $[H:\widehat{H}]=2$.

\begin{thm}
$\widehat{H}(\vec n)$ is separable in $F$
when $\vec n$ is multiplicative.
\end{thm}
One can follow the sketch of proof of Theorem~\ref{thm:H_separable} to show directly that $\widehat H=\cap_k \widehat H_k$, but we instead obtain it as a consequence of Theorem~\ref{thm:H_separable}.
\begin{proof}
    We have $[H:\widehat{H}] = 2$.
    And 
    $H = \cap_k H_k$ by Theorem~\ref{thm:H_separable}.
    Let $\widehat{\mathcal{H}}_k$ be the  double cover of $\mathcal{H}_k$  illustrated in Figure~\ref{fig:Hihat}, and $\widehat{H}_k = \pi_1\widehat{\mathcal{H}}_k$.
    To show separability of $\widehat{H}$ we will show $\widehat{H} = \cap_k \widehat{H}_k$.

Since $\widehat{\mathcal H} \rightarrow \mathcal F$ factors as $\widehat{\mathcal H }\rightarrow \widehat{\mathcal H_k} \rightarrow \mathcal F$
we have $\widehat{H} \subset  \widehat{H}_k$.
Consequently,  $\widehat{H} \subseteq \cap_k \widehat{H}_k$.
For the other direction, consider the inclusions $\widehat{H} \subseteq \cap_k \widehat{H}_k \subseteq \cap_k H_k = H$.
Observe that $\cap_k \widehat{H}_k \neq H$
since 
$b^{n_0} \in H$ but $b^{n_0} \notin \cap_k \widehat{H}_k$.
Since $[H:\widehat{H}] = 2$ there is no properly intermediate subgroup, so $H = \cap_k H_k$.
\end{proof}

Consider the double $G = F * _{\widehat{H}} \underline{F}$.
The residual finiteness of $G$
holds by Lemma~\ref{lem:rf_double}.

There is a Klein-bottle subgroup of $G$ for each $i$:
 $$K_i = \langle a^{-i} b^{n_i} a^i, \ \underline{a}^{-i}\underline{b}^{n_i}\underline{a}^{i} \rangle.$$

Let $Y = H \underset{\widehat{H}}{*} \underline{H}$.  Lemma~\ref{lem:NFT_and_doubles} gives an injection $Y\hookrightarrow G$. Regard $Y$  as a subgroup of $G$.

Let $\widehat{Y} = C_G(z)$ where $z = b^{-n_0} \underline{b}^{n_0}$.
Then by Lemma~\ref{lem:the_centralizer}, $\widehat{Y} \subset Y$ and $[Y:\widehat{Y}] = 2$.

There is a  torus subgroup $T_i \subset K_i \subset Y$ for each $i$.
Here $T_i=\widehat Y \cap K_i$, is the index~2 subgroup of $K_i$
acting by translations on the axis of $z$
in the proof of Lemma~\ref{lem:the_centralizer}.

$$T_i 
\ = \ 
\langle a^{-i}b^{2n_i}a^i
, \
a^{-i}b^{n_i}a^i
{\underline a}^{-i}{\underline b}^{-n_i}{\underline a}^i 
\
\rangle.$$

The groups we are interested in are the following doubles:
$$D \ = \ D(\vec{n}) \ = \ G  \underset{\widehat Y}{*} \underline{G}.$$

\begin{thm}\label{thm:D_rf}
$\widehat Y \subset G$ is separable.

Consequently
$D=G\underset{\widehat Y}{*} \underline{G}$ is residually finite.
\end{thm}

\begin{proof}
The second statement follows from the first by Lemma~\ref{lem:rf_double}.

The first statement holds by Lemma~\ref{lem:the_centralizer} and Corollary~\ref{cor:centralizer_separable}.
\end{proof}

Let $P_i=\langle K_i , \underline K_i \rangle$ be the $i$-th copy of the Promislow group in $D$.
Indeed, each
 $K_i \underset{T_i}{*} \underline{K}_i \rightarrow  G \underset{\widehat Y}{*} G$ is  injective by Lemma~\ref{lem:NFT_and_doubles}.
Hence $P_i \cong K_i\underset{T_i}{*} \underline K_i \cong P$.

We have now arrived at our main goal:

\begin{thm}
Suppose $\vec n$ has the property that for each  $r\geq 1$, we have $r!$ divides some $n_i$.
Then $D=D(\vec n)$ persistently contains $P$.

In particular, this holds for $n_i=2\cdot(i!)$.
\end{thm}
\begin{proof}
For an index~$q$ subgroup $C\subset D$,
we have $d^{q!}\in C$ for each $d\in D$.
Hence $P_q \subset  C$.
\end{proof}

\section{A few remarks}
\begin{rem}[nonlinear]
Each $D=D(\vec n)$ is non-linear.
In fact, the subgroup $G\subset D$  is not linear.
Indeed, conjugating the left factor of $G$ by $a$ and the right factor of $G$ by $\underline a$,
one obtains an endomorphism that is not ``eventually injective''
(the same works for the double).
Hence these groups are not linear 
\cite{WiseEventuallyInjective2025}. See Problem~\ref{prob:linear_virtually_UPP}.
\end{rem}

\begin{rem}[amenable]
Finitely generated
linear
amenable groups are virtually solvable by Tits' alternative.
And in characteristic 0, they are
 virtually torsion-free by 
 Selberg \cite{Alperin1987}.
They are  virtually (torsion-free-nilpotent)-by-abelian, by a theorem of Mal\'cev \cite{Malcev1951}.
One can then show that they are virtually locally-indicable. 
It follows that they are virtually left-orderable, and hence virtually have  the UPP.

\end{rem}

\begin{rem}[uncountable]\label{rem:uncountable}
If we vary the sequence $\vec n$,
then we obtain  uncountably many examples that are residually finite but not virtually UPP.
They are distinguished up to isomorphism using  a homology computation,
as in \cite{ChongWise2022}.

One can also apply 
\cite{Bowditch1998} to see that there are uncountably many quasi-isometry classes. 
\end{rem}

\noindent {\bf Acknowledgement:} We are grateful to the referee for many helpful corrections.

\bibliographystyle{plain} 
\bibliography{wise} 

@article {Alperin1987,
    AUTHOR = {Alperin, Roger C.},
     TITLE = {An elementary account of {S}elberg's lemma},
   JOURNAL = {Enseign.Math.},
  FJOURNAL = {L'Enseignement Math\'ematique. Revue Internationale. 2e
              S\'erie},
    VOLUME = {33},
      YEAR = {1987},
    NUMBER = {3-4},
     PAGES = {269--273},
      ISSN = {0013-8584},
   MRCLASS = {20H20},
  MRNUMBER = {925989},
MRREVIEWER = {A.\ E.\ Zalesski\u i},
}

@article {Malcev1951,
    AUTHOR = {Mal\cprime{c}ev, A. I.},
     TITLE = {On some classes of infinite soluble groups},
   JOURNAL = {Mat. Sbornik N.S.},
  FJOURNAL = {Mat. Sbornik N.S.},
    VOLUME = {28/70},
      YEAR = {1951},
     PAGES = {567--588},
   MRCLASS = {20.0X},
  MRNUMBER = {43088},
MRREVIEWER = {K.\ A.\ Hirsch},
}

@misc{Ng2025,
      title={Virtual First Betti Number of {GGS} Groups}, 
      author={Andrew Ng},
      year={2025},
      eprint={2505.23269},
      archivePrefix={arXiv},
      primaryClass={math.GR},
      url={https://arxiv.org/abs/2505.23269} 
}

@article {BolerEvans1973,
    AUTHOR = {Boler, James and Evans, Benny},
     TITLE = {The free product of residually finite groups amalgamated along
              retracts is residually finite},
   JOURNAL = {Proc. Amer. Math. Soc.},
  FJOURNAL = {Proceedings of the American Mathematical Society},
    VOLUME = {37},
      YEAR = {1973},
     PAGES = {50--52},
      ISSN = {0002-9939,1088-6826},
   MRCLASS = {20E25},
  MRNUMBER = {306329},
MRREVIEWER = {L.\ Neuwirth},
       DOI = {10.2307/2038704},
       URL = {https://doi.org/10.2307/2038704},
}

@article {Gardam2021,
    AUTHOR = {Gardam, Giles},
     TITLE = {A counterexample to the unit conjecture for group rings},
   JOURNAL = {Ann. of Math. (2)},
  FJOURNAL = {Annals of Mathematics. Second Series},
    VOLUME = {194},
      YEAR = {2021},
    NUMBER = {3},
     PAGES = {967--979},
      ISSN = {0003-486X,1939-8980},
   MRCLASS = {20C07},
  MRNUMBER = {4334981},
MRREVIEWER = {E.\ Formanek},
       DOI = {10.4007/annals.2021.194.3.9},
       URL = {https://doi.org/10.4007/annals.2021.194.3.9},
}

@article {Bowditch1998,
    AUTHOR = {Bowditch, B. H.},
     TITLE = {Continuously many quasi-isometry classes of {$2$}-generator
              groups},
   JOURNAL = {Comment. Math. Helv.},
  FJOURNAL = {Commentarii Mathematici Helvetici},
    VOLUME = {73},
      YEAR = {1998},
    NUMBER = {2},
     PAGES = {232--236},
      ISSN = {0010-2571,1420-8946},
   MRCLASS = {20F32 (20F05)},
  MRNUMBER = {1611695},
MRREVIEWER = {Ilya\ Kapovich},
       DOI = {10.1007/s000140050053},
       URL = {https://doi.org/10.1007/s000140050053},
}

@article {Bowditch2000,
    AUTHOR = {Bowditch, B. H.},
     TITLE = {A variation on the unique product property},
   JOURNAL = {J. London Math. Soc. (2)},
  FJOURNAL = {Journal of the London Mathematical Society. Second Series},
    VOLUME = {62},
      YEAR = {2000},
    NUMBER = {3},
     PAGES = {813--826},
      ISSN = {0024-6107,1469-7750},
   MRCLASS = {20F65},
  MRNUMBER = {1794287},
MRREVIEWER = {Thomas\ Delzant},
       DOI = {10.1112/S0024610700001307},
       URL = {https://doi.org/10.1112/S0024610700001307},
}

@article {KonkeRaimbaultDunfield_2016,
    AUTHOR = {Kionke, Steffen and Raimbault, Jean},
     TITLE = {On geometric aspects of diffuse groups},
      NOTE = {With an appendix by Nathan Dunfield},
   JOURNAL = {Doc. Math.},
  FJOURNAL = {Documenta Mathematica},
    VOLUME = {21},
      YEAR = {2016},
     PAGES = {873--915},
      ISSN = {1431-0635,1431-0643},
   MRCLASS = {20F65 (20-04 20H15 22E40 57M07)},
  MRNUMBER = {3548136},
MRREVIEWER = {Nansen\ Petrosyan},
}

@article {WiseEventuallyInjective2025,
    AUTHOR = {Wise, Daniel T.},
     TITLE = {On eventually injective endomorphisms},
   JOURNAL = {J. Comb. Algebra},
  FJOURNAL = {Journal of Combinatorial Algebra},
    VOLUME = {9},
      YEAR = {2025},
    NUMBER = {1-2},
     PAGES = {121--127},
      ISSN = {2415-6302,2415-6310},
   MRCLASS = {20H20 (20E06 20E26)},
  MRNUMBER = {4876493},
       DOI = {10.4171/jca/93},
       URL = {https://doi.org/10.4171/jca/93},
}

@article {ChongWise2022,
    AUTHOR = {Chong, Hip Kuen and Wise, Daniel T.},
     TITLE = {An uncountable family of finitely generated residually finite
              groups},
   JOURNAL = {J. Group Theory},
  FJOURNAL = {Journal of Group Theory},
    VOLUME = {25},
      YEAR = {2022},
    NUMBER = {2},
     PAGES = {207--216},
      ISSN = {1433-5883},
   MRCLASS = {20E26 (20E22 20F05)},
  MRNUMBER = {4388367},
MRREVIEWER = {Valeriy G. Bardakov},
       DOI = {10.1515/jgth-2021-0094},
       URL = {https://doi-org.proxy3.library.mcgill.ca/10.1515/jgth-2021-0094},
}

@article {RipsSegev87,
    AUTHOR = {Rips, Eliyahu and Segev, Yoav},
     TITLE = {Torsion-free group without unique product property},
   JOURNAL = {J. Algebra},
  FJOURNAL = {Journal of Algebra},
    VOLUME = {108},
      YEAR = {1987},
    NUMBER = {1},
     PAGES = {116--126},
      ISSN = {0021-8693},
     CODEN = {JALGA4},
   MRCLASS = {20F06 (20C07)},
  MRNUMBER = {MR887195 (88g:20071)},
MRREVIEWER = {Stephen J. Pride},
}

@preamble{
   "\def\cprime{$'$} "
}

@preamble{
   "\def\polhk#1{\setbox0=\hbox{#1}{\ooalign{\hidewidth
    \lower1.5ex\hbox{`}\hidewidth\crcr\unhbox0}}} "
}

@article {Baumslag63b,
   AUTHOR = {Baumslag, Gilbert},
    TITLE = {On the residual finiteness of generalised free products of
             nilpotent groups},
  JOURNAL = {Trans. Amer. Math. Soc.},
   VOLUME = {106},
     YEAR = {1963},
    PAGES = {193--209},
  MRCLASS = {20.52},
 MRNUMBER = {26 #2489},
   MRREVR = {Hanna Neumann},
}

@ARTICLE{Stallings83,
   author = {Stallings, John R.},
   title = {Topology of finite graphs},
   journal = {Invent. Math.},
   year = {1983},
   volume = {71},
   number = {3},
   pages = {551--565},
}

@article {WiseDescending99,
   AUTHOR = {Wise, Daniel T.},
    TITLE = {A continually descending endomorphism of a finitely generated
             residually finite group},
  JOURNAL = {Bull. London Math. Soc.},
 FJOURNAL = {The Bulletin of the London Mathematical Society},
   VOLUME = {31},
     YEAR = {1999},
   NUMBER = {1},
    PAGES = {45--49},
     ISSN = {0024-6093},
    CODEN = {LMSBBT},
}

@BOOK{LS77,
   author = {Lyndon, Roger C. and Schupp, Paul E.},
   title = {Combinatorial group theory},
   publisher = {Springer-Verlag},
   year = {1977},
   address = {Berlin},
   note = {Ergebnisse der Mathematik und ihrer Grenz\-gebiete, Band 89},
}

\end{document}